\documentclass [10pt,reqno]{amsart}
\usepackage{ucs}
\usepackage{amsmath, amssymb, amscd}
\usepackage[applemac]{inputenc}
\usepackage{amsmath}
\usepackage{amssymb}
\usepackage{graphicx}
\usepackage[usenames,dvipsnames,svgnames,table]{xcolor}

\newtheorem{theorem}{Theorem}[section]
\newtheorem{definition}[theorem]{Definition}
\newtheorem{lemma}[theorem]{Lemma}
\newtheorem{prop}[theorem]{Proposition}
\newtheorem{corollary}[theorem]{Corollary}

\newtheorem{hypothesis}[theorem]{Hypothesis}
\newtheorem{example}[theorem]{Example}

\newtheorem{remark}[theorem]{Remark}
\newtheorem{question}[theorem]{Question}

\newcommand{\abs}[1]{\lvert#1\rvert}

\renewcommand{\epsilon}{\varepsilon}

\DeclareMathAlphabet{\mathpzc}{OT1}{pzc}{m}{it}
\usepackage{mathrsfs}

\newcommand{\Z}{\mathbb{Z}}
\newcommand{\C}{\mathbb{C}}

\newcommand{\R}{\mathbb{R}}

\renewcommand{\qed}{$\hfill \square$ \bigskip \\}
\renewcommand{\phi}{\varphi}

\newcommand{\Hom}{\text{Hom}}

\renewcommand{\det}{\text{det}}

\renewcommand{\i}{{\bf i}}
\renewcommand{\j}{{\bf j}}
\renewcommand{\k}{{\bf k}}

\newcommand{\su}{\mathfrak{su}}

\begin{document}
\thispagestyle{empty}
\title{A class of knots with simple $SU(2)$-representations} 
\author[Raphael Zentner ]{Raphael Zentner} 
%\date{February 2012}
\address {Fakult\"at f\"ur Mathematik \\
Universit\"at Regensburg\\
Germany}
\email{raphael.zentner@mathematik.uni-regensburg.de}

\begin{abstract}
We call a knot in the 3-sphere $SU(2)$-simple if all representations of the fundamental group of its complement which map a meridian to a trace-free element in $SU(2)$ are binary dihedral.  This is a generalization of being a 2-bridge knot. Pretzel knots with bridge number $\geq 3$ are not $SU(2)$-simple. We provide an infinite family of knots $K$ with bridge number $\geq 3$ which are $SU(2)$-simple. 

One expects the instanton knot Floer homology $I^\natural(K)$ of a $SU(2)$-simple knot to be as small as it can be -- of rank equal to the knot determinant $\det(K)$. In fact, the complex underlying $I^\natural(K)$ is of rank equal to $\det(K)$, provided a genericity assumption holds that is reasonable to expect. Thus formally there is a resemblance to strong L-spaces in Heegaard Floer homology. 
For the class of $SU(2)$-simple knots that we introduce this formal resemblance is reflected topologically: The branched double covers of these knots are strong L-spaces. In fact, somewhat surprisingly, these knots are alternating. However, the Conway spheres are hidden in any alternating diagram.

With the methods we use, we obtain the result that an integer homology 3-sphere which is a graph manifold always admits irreducible representations of its fundamental group. This makes use of a non-vanishing result of Kronheimer-Mrowka.
\end{abstract}

\maketitle

\section{Introduction}

The purpose of this paper is to study knots that are particularly simple with respect to the $SU(2)$-representation variety of the fundamental group of the knot complement.

\begin{definition}
A knot $K$ is called $SU(2)$-simple if the space $R(K;\i)$ of representations of the fundamental group in $SU(2)$, defined in Section \ref{representations} below, contains only representations that are binary dihedral. 
\end{definition}

Two-bridge knots are $SU(2)$-simple. The results of the author in \cite{Z} show that pretzel knots which are not 2-bridge knots are not $SU(2)$-simple knots. The author was tempted to believe that 2-bridge knots were the only $SU(2)$-simple knots. However, in this paper we give a large class of $SU(2)$-simple knots. 

This makes use of the following fact: 
\setcounter{section}{3}
\setcounter{theorem}{0}
\begin{prop}
	Let $K$ be a knot. If $\pi_1(\Sigma_2(K))$ has only cyclic $SO(3)$ representations, then $R(K;\i)$ is $SU(2)$-simple. Here $\Sigma_2(K)$ denotes the branched double cover of the knot $K$.
\end{prop}

The interest in $SU(2)$-simple knots comes from the fact that they are expected to be particularly simple with respect to Kronheimer-Mrowka's instanton knot Floer homology $I^\natural(K)$ \cite{KM_ss}. In fact, the following Proposition follows from the relation of the underlying chain complex to $R(K;\i)$ and the relationship of instanton knot Floer homology to the Alexander polynomial, established independently by  Kronheimer and Mrowka in \cite{KM_Alex} and Lim in \cite{Lim}. 

\setcounter{section}{7}
\setcounter{theorem}{2}
\begin{prop}
If a knot $K$ is $SU(2)$-simple and satisfies the genericity hypothesis \ref{hypothesis regularity}, then its instanton Floer chain complex $CI^\natural_\pi(K)$ has no non-zero differentials and is of total rank $\det(K)$. In particular, the total rank of reduced instanton knot Floer homology $I^\natural(K)$ is also equal to $\det(K)$. 
\end{prop}

Denoting $Y(T(p,q))$ the complement of a tubular neighborhood of the torus knot $T(p,q)$, we may glue $Y(T(p,q))$ and $Y(T(r,s))$ together along their boundary torus in such a way that a meridian of the first torus knot is mapped 
to a Seifert fibre of the second and vice versa. It is a result of Motegi \cite{Motegi} that these 3-manifolds $Y= Y(T(p,q),T(r,s))$ are $SO(3)$-cyclic, i.e. admit only cyclic $SO(3)$ representations of their fundamental group. 
Using the concept of strongly invertible knots, we obtain 

\setcounter{section}{4}
\setcounter{theorem}{13}
\begin{theorem}
	The 3-manifold $Y(T(p,q),T(r,s))$ comes with an involution with quotient $S^3$. It is a branched double cover of some knot or 2-component link $L(T(p,q),T(r,s))$ in $S^3$, well defined up to mutation by the involution on either side of the essential torus in $Y(T(p,q),T(r,s))$. If $pqrs-1$ is odd then $L(T(p,q),T(r,s)$ is a knot. If in addition both $T(p,q)$ and $T(r,s)$ are non-trivial torus knots, then the knot $L(T(p,q),T(r,s)$ is $SU(2)$-simple, but is not a 2-bridge knot.
\end{theorem}

We give an explicit description of the knots $L(T(p,q),T(r,s))$ as a decomposition of two tangles in Section \ref{diagrams} below. In fact, each tangle is explicitly described in Theorem \ref{tangle_standard_diagram}. Somewhat to our surprise, we have obtained 

\setcounter{section}{6}
\setcounter{theorem}{4}
\begin{theorem}
The knots $L(T(p,q),T(r,s))$ are alternating. The Conway sphere giving rise to the essential torus in the branched double cover is not visible in any alternating diagram of the link. 
\end{theorem}
The method of the proof is applicable to a larger class of knots and links. This may be compared to the method of Greene and Levine in \cite{Greene_Levine}. 

\setcounter{section}{8}
\setcounter{theorem}{0}
\begin{corollary}
The 3-manifolds $Y(T(p,q),T(r,s))$ are strong L-spaces (in the sense of Heegaard Floer homology). 
\end{corollary}

It is therefore tempting to ask whether the class of $SU(2)$-simple knots consists exclusively of knots whose branched double cover is a strong L-space. This would set up some correspondence between instanton knot Floer homology and Heegaard Floer homology. It is by far not true, however, that any alternating knot is $SU(2)$-simple in the sense given above. 
\\

Finally, as a Corollary of Proposition \ref{folk} and by using a non-vanishing result of Kronheimer-Mrowka \cite{KM_sutures} and further results of Bonahon-Siebenmann \cite{Bonahon-Siebenmann}, we obtain the following

\setcounter{section}{9}
\setcounter{theorem}{1}
\begin{corollary}
Let $Y$ be an integer homology sphere which is a graph manifold. Then there is an irreducible representation $\rho: \pi_1(Y) \to SU(2)$. 
\end{corollary}

\setcounter{section}{1}

\section*{Acknowledgements}
The author would like to thank Stefan Friedl, Paul Kirk, Jianfeng Lin, Ciprian Manoles\-cu, Nikolai Saveliev, and Liam Watson for stimulating conversation. He is particularly grateful to Michel Boileau for sharing his expertise. The author is thankful for support by the SFB `Higher Invariants' at the University of Regensburg, funded by the Deutsche Forschungsgesellschaft (DFG). Furthermore, he would like to thank Joshua Greene, Ilya Kofman, Tye Lidman and Brendan Owens for comments on the first version of this article which appeared on the arXiv. He is also thankful for comments made by the referee. 

\section{Meridian-traceless $SU(2)$ representations}\label{representations}
Throughout we shall feel free to identify the group $SU(2)$ with the group of unit quaternions. By $\i, \j, \k$ we denote the unit quaternions.

\begin{definition}
   Let $K$ be a knot in $S^3$. We assume some base-point fixed in its complement. Let $m$ be a closed based path in $S^3 \setminus K$ that yields a generator of $H_1(S^3 \setminus K;\Z)$. We shall call the following space the representation space of meridian-traceless $SU(2)$ representations. 
   \begin{equation*}
   	R(K;\i) := \{ \rho \in \Hom(\pi_1(S^3 \setminus K),SU(2)) \, | \, 			\rho(m) \sim \i \} \ ,
   \end{equation*}
where $\i$ denotes a purely imaginary quaternion (all of which are conjugate), and where $\rho(m) \sim \i$ denotes the requirement that $\rho(m)$ is conjugate in $SU(2)$ to this element. We shall also denote this space $R(K,SU(2);\i)$ if we want to make the Lie group $SU(2)$ explicit. 
As all based meridians are conjugate, this definition is independent of the choice of meridian.

Likewise, we define 
   \begin{equation*}
   	R(K, SO(3);I) := \{ \rho \in \Hom(\pi_1(S^3 \setminus K),SO(3)) \, | \, 			\rho(m) \sim I \} \ ,
   \end{equation*}
where $I$ denotes an element of order $2$ in $SO(3)$ (all of which are conjugate).
\end{definition}

Calling the representations of $R(K,SU(2);\i)$ {\em meridian-traceless} is sensible because under the standard isomorphism of $SU(2)$ with the unit quaternions the traceless elements correspond precisely to the purely imaginary quaternions.

\section{Binary dihedral representations and the double branched cover}
Let $K$ be a knot, and let $\Sigma_2(K)$ denote the double branched cover of the knot.
Recall that we have a short exact sequence of groups 
\begin{equation*}
	1 \to \Z/2 \to SU(2) \stackrel{\pi}{\to} SO(3) \to 1 \, .
\end{equation*} 
We shall say that a representation of a group $G$ in a group $H$ is `$H$-abelian' or `has only abelian $H$ representations' if its image in $H$ is contained in an abelian subgroup of $H$, and similarly `$H$-cyclic'. In this article the group $G$ will always be some fundamental group, and $H$ will bei either $SO(3)$ or $SU(2)$.   
A representation in $SO(3)$ is called dihedral if its image is contained in a dihedral subgroup of $SO(3)$ -- a group generated by rotations fixing a globally fixed plane, and reflections of that plane. A representation in $SU(2)$ is called binary dihedral if its image in $SO(3)$ is dihedral.

\begin{prop}\label{folk}
	If $\pi_1(\Sigma_2(K))$ has only cyclic $SO(3)$ representations, then $R(K;\i)$ admits only binary dihedral meridian-traceless $SU(2)$ representations, i.e. $K$ is $SU(2)$-simple.
\end{prop}

In fact, the condition `only cyclic $SO(3)$ representations' can be replaced by `only cyclic $SU(2)$ representations' by the following 
\begin{lemma}\label{cyclic vs abelian}
Let $Y$ be a closed 3-manifold such that $H_1(Y;\Z/2)=0$. Then the following are equivalent:
\begin{enumerate}
	\item The group $\pi_1(Y)$ has only abelian $SO(3)$ representations.
	\item The group $\pi_1(Y)$ has only abelian $SU(2)$ representations.
	\item The group $\pi_1(Y)$ has only cyclic $SO(3)$ representations.
	\item The group $\pi_1(Y)$ has only cyclic $SU(2)$ representations.
\end{enumerate}
\end{lemma}
\begin{remark}
	An infinite abelian subgroup of $SO(3)$ or $SU(2)$ needs not to be cyclic.
\end{remark}

\begin{proof}[Proof of the Lemma]
The obstruction to lifting a representation $\rho: \pi_1(Y) \to SO(3)$ to a representation $\tilde{\rho}: \pi_1(Y) \to SU(2)$ is an element $w_2(\rho) \in H^2(Y;\Z)$. By Poincar\'e duality $H^2(Y;\Z/2) \cong H_1(Y;\Z/2)$, and by our assumption this group vanishes. Therefore, there is no obstruction to lifting $\rho$ to $SU(2)$.

Because of this lifting behavior, if $\pi_1(Y)$ has only abelian $SU(2)$ representations, it can only have abelian $SO(3)$ representations, and if it has only cyclic $SU(2)$ representations, it has only cyclic $SO(3)$ representations. 

Conversely, suppose that $\pi_1(Y)$ has only abelian representations in $SO(3)$. The only abelian subgroups of $SO(3)$ are given by subgroups of rotations with a fixed axis, and subgroups isomorphic to $K=\Z/2 \times \Z/2$ where the three non-trivial elements are given by rotations by angle $\pi$ along three axes which are all pairwise perpendicular. Subgroups of the first kind lift to abelian subgroups of $SU(2)$, so a representation factoring through an abelian subgroup of the first kind lifts to an abelian representation in $SU(2)$. The preimage of $K$ in $SU(2)$ is a non-abelian subgroup of $SU(2)$ with $8$ elements, namely, the quaternionic group $\{ \pm 1, \pm \i, \pm \j, \pm \k \}$ in quaternion notation. However, any representation of $\pi_1(Y)$ with image in $K$ in fact factors through the abelianizabelianisationation $H_1(Y;\Z)$ which has odd order, hence any element has odd order, and therefore such a representation factors in fact through the trivial subgroup of $SO(3)$, and hence has a lift to an abelian subgroup of $SU(2)$.

Similarly as before, using again that $H_1(Y;\Z)$ has odd order and the chinese remainder theorem, one sees that $\pi_1(Y)$ has only cyclic $SU(2)$ representations if it has only cyclic $SO(3)$ representations. 

Clearly, if a representation of $\pi_1(Y)$ has only cyclic $SO(3)$ representations, then it has only abelian $SO(3)$ representations. Conversely, again because $H_1(Y;\Z)$ is odd, any abelian $SO(3)$ representation is in fact cyclic. 
\end{proof}

\begin{proof}[Proof of the Proposition]
Let us denote by $G_K$ the fundamental group of the knot complement, and by $G_{K,m^2}$ the $\pi$-orbifold fundamental group defined by
\[
	G_{K,m^2} := G_K / \langle \langle m^2 \rangle \rangle \ ,
\]
where $\langle \langle m^2 \rangle \rangle$ denotes the normal subgroup generated by the square of the meridian. This is a powerful group that is sufficiently strong to determine the bridge numbers of Montesinos knots, see \cite{Boileau_Zieschang}. 

The projection $\pi$ induces a well-defined map
\begin{equation*}
	\pi_*: R(K,SU(2);\i) \to R(K,SO(3);I) \ .
\end{equation*}
In fact, $\pi$ takes purely imaginary quaternions to rotations with angle $\pi$, so elements of order 4 get mapped to elements of order 2. 

It is shown By Collin and Saveliev in \cite[Proposition 3.4]{Collin-Saveliev} (in a more general situation than the one we need here) that this map $\pi_*$ is a double cover ramified along the binary dihedral representations. (We outline the proof in this simplified context for the sake of completeness: The space $S^3 \setminus K$ has the homology of a circle, and therefore there is no obstruction in lifting any $SO(3)$ representation to a $SU(2)$ representation here. Therefore, the map $\pi_*$ is onto.
The space $\Hom(G_K,\Z/2) \cong \Z/2$ acts on the space of $SU(2)$ representations, and any two differing by this action yield the same $SO(3)$ representation. Conversely, any two yielding the same $SO(3)$ representation differ by such a central representation. Therefore, the map is at most 2 to 1. One easily checks that the fixed points of this involution are precisely the binary dihedral representations. )

An intermediate conclusion is that $R(K,SU(2);\i)$ contains only binary dihedral representations if and only if $R(K,SO(3);I)$ contains only dihedral representations. On the other hand, as in $R(K,SO(3);I)$ meridians are mapped to elements of order 2, any such representation from $G_K$ to $SO(3)$ factors through $G_{K,m^2}$. Hence $R(K,SU(2);\i)$ has only binary dihedral representations in $SU(2)$ if and only if the orbifold fundamental group $G_{K,m^2}$ has only dihedral representations in $SO(3)$. \\

The orbifold fundamental group $G_{K,m^2}$ fits into a short exact sequence of groups
\begin{equation}\label{orbifold group exact sequence}
	1 \to \pi_1(\Sigma_2(K)) \to G_{K,m^2} \to \Z/2 \to 1 \, ,
\end{equation}
as can be proved easily with the Seifert-van Kampen theorem.
This sequence splits by mapping the non-trivial element of $\Z/2$ to the meridian $m$. Hence $G_{K,m^2}$ is a semi-direct product of $\Z/2$ acting on $\pi_1(\Sigma_2(K))$. 

Let us now consider a representation $\rho: G_{K,m^2} \cong \pi_1(\Sigma_2(K)) \rtimes \Z/2 \to SO(3)$ that is cyclic when restricted to $\pi_1(\Sigma_2(K))$. Therefore, the image of $\pi_1(\Sigma_2(K))$ is a finite subgroup of $SO(3)$ that consists of rotations all of which have the same rotation axis $z$. The image $I$ of the generator $m$ of $\Z/2$ has to act on this finite cyclic subgroup of $SO(3)$. There are only two possibilities: The rotation axis of $I$ coincides with $z$, in which case the entire image $\rho(G_{K,m^2})$ is cyclic, or the rotation axis of $I$ is perpendicular to $z$, in which case the image $\rho(G_{K,m^2})$ is a dihedral group. 
\end{proof}

%Moving the base-point in $\Sigma_2(K)$ into the fixed point set of the involution $\sigma$ coming from the construction of the double branched cover, we can think of $\sigma$ 

\section{Strongly invertible knots}
A knot $K \subseteq \R^3 \subseteq S^3$ is called strongly invertible if there is a straight line $A$ in $\R^3$ (extending to an $S^1$ in $S^3$) such that rotation by angle $\pi$ around the axis $A$ preserves $K$, and such that $K$ has precisely two intersection points with $A$. The Figure \ref{strongly_invertible}
 below illustrates the fact that the torus knot $T(3,4)$ is strongly invertible.

Necessarily this involution $\sigma$ reverses the orientation of the knot $K$, so $K$ is an invertible knot. Notice also that $S^3/\sigma$ is homeomorphic to $S^3$. The following result on strongly invertible knots is standard, see \cite{Montesinos_Orsay}. We include a short proof for the sake of completeness. 

\begin{figure}[h!]
\caption{The torus knot $T(3,4)$ is strongly invertible}
\label{strongly_invertible}
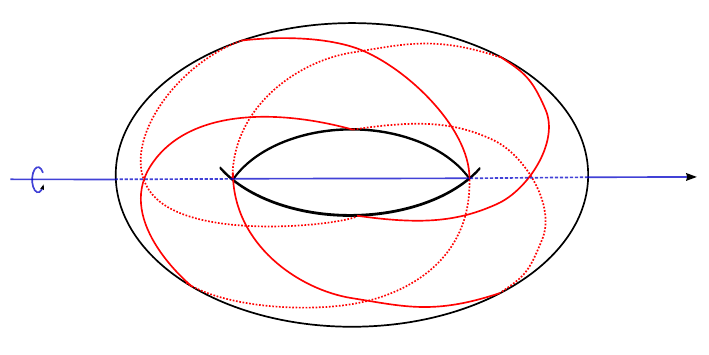
\end{figure}

\begin{lemma}\label{knot complement quotient}
Let $K$ be a strongly invertible knot with axis $A$ and involution $\sigma$ given by rotation by $\pi$ around $A$. Let $N(K)$ be an open tubular neighborhood of $K$ which is invariant under $\sigma$. Let $Y(K):= S^3 \setminus N(K)$. Then the quotient $Y(K)/\sigma$ is homeomorphic to a 3-ball. 
\end{lemma}
\begin{proof}
	The quotient of the boundary torus $\partial Y(K)$ by $\sigma$ is a double cover of the 2-sphere, branched along four points. Therefore the manifold $Y(K)/\sigma$ is a 3-manifold with boundary a 2-sphere, and which is a submanifold of $S^3/\sigma \cong S^3$. However, any 2-sphere embedded in $S^3$ splits the 3-sphere into two balls, by Schö\"onflies' theorem.
\end{proof}

The following Lemma is known to experts, but the author is unaware of a reference where a proof is given. At the level of involutions of the fundamental group of torus knot complements, this can be traced back to Schreier \cite{Schreier}. We give a short geometric proof for the sake of completeness. 

\begin{lemma}
	The torus knots $T_{p,q}$ are strongly invertible. 
\end{lemma}
\begin{proof}
	We think of a standard embedded torus inside $\R^3$ that is perforated by a skewer in four points, such that rotation by $\pi$ around the skewer yields a symmetry of the torus. When we suitably identify this torus with $\R^2 / \Z^2$, the involution becomes
a point reflection in $(\frac{1}{2},\frac{1}{2})$ in the fundamental domain $[0,1]\times[0,1]$, with the four fixed points being the classes of the four points 
\[
\left(0,0\right),\left(0,\frac{1}{2}\right),\left(0,\frac{1}{2}\right),\left(\frac{1}{2},\frac{1}{2}\right) \, .
\]
We obtain the torus knot $T(p,q)$ by drawing a straight line  in the plane $\R^2$, starting at $(0,0)$, and passing through $(q,p)$. It is now an elementary arithmetic exercise to check that it only passes through one other fixed point of the involution $\sigma$, using that $p$ and $q$ are coprime. 
\end{proof}

\begin{definition}
	We denote by $\tau(K)$ the tangle
	\[
		(Y(K)/\sigma, (A \cap Y(K))/\sigma))
	\]
	obtained from the strongly invertible knot $K$ with involution $\sigma$ around the axis $A$.
\end{definition}

The Figures  \ref{quotient}
 and \ref{quotient_isotopic} 
below are both pictures of the tangle $\tau(T(3,4))$. In both pictures the region shaded in light red indicates the quotient of the tubular neighborhood $N(T(3,4))$, a 3-ball and trivial tangle, and the tangle $\tau(K)$ is the complement. Figure \ref{quotient_isotopic} is obtained from Figure \ref{quotient} by an isotopy of $S^3$ which maps a standard torus around which $T(3,4)$ is winding, modulo the involution, to the `pillowcase' where the four fixed points are indicated by corners. From these pillowcase pictures it is straightforward how to draw tangles for other torus knots. 

\begin{figure}[h!]
\caption{The quotient tangle $\tau(T(3,4))$, half of the image of a Seifert fibre is indicated by the green line}
\label{quotient}
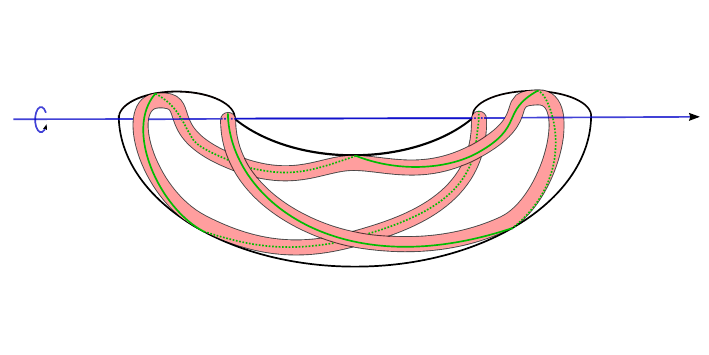
\end{figure}

\begin{figure}[h!]
\caption{Picture of the tangle $\tau(T(3,4))$ after an isotopy that brings the solid torus quotient into the standard `pillowcase' shape. The green circle is the image of a Seifert-fibre on the boundary, the red circle is the image of a meridian on a tubular neighborhood of $T(3,4)$, indicated in light red. The blue squares indicate where the axis (in blue) is entering and leaving the quotient of the tubular neighborhood of the torus knot $T(3,4)$.}
\label{quotient_isotopic}
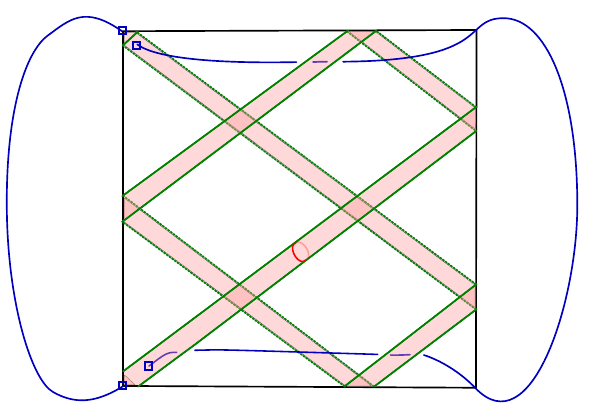
\end{figure}

By a longitude of a knot $K$ we understand a curve parallel to $K$ in its complement which is homologically trivial, or, equivalently, which has linking number $0$ with the knot. By a meridional disk of $K$ we understand a properly embedded disk in a tubular neighborhood of $K$ which has one transversal self-intersection point with $K$. A meridian is the boundary of such a disk. 

We orient boundaries of oriented manifolds with boundary by the `outward normal first' convention.
We choose an orientation of the knot $K$. This determines an orientation of a longitude. It also determines an orientation of a meridian: There is an orientation on the meridional disk such that the intersection of the disk with the knot is positive. This orientation of the disk induces an orientation on the meridian. We give $S^1$ the natural counterclockwise orientation, and we orient products canonically.
\\

\begin{lemma} \label{parametrisation}
For any strongly invertible knot $K$ with involution $\sigma$ and invariant tubular neighborhood $N(K)$ there is an orientation-preserving homeomorphism 
$h_K: \partial Y(K) \to S^1 \times S^1$, where $\partial Y(K)$ denotes the boundary of $Y(K) = S^3 \setminus N(K)$, such that 
\begin{itemize}
	\item[(i)] the circles $\{ pt \} \times S^1$ correspond to meridians of $K$ with matching orientations,
	\item[(ii)] the circles $ S^1 \times \{ pt \} $ correspond to longitudes of $K$ with matching orientations, such that
	\item[(iii)] the restriction of the involution translates to
	 $$h_K \circ \sigma|_{\partial Y(K)} \circ h_K^{-1} = \begin{pmatrix} -1 & 0 \\ 0 & -1 \end{pmatrix} \, , $$ where a map $S^1 \times S^1 \to S^1 \times S^1$ associated with a matrix is the map induced by the linear map $\R^2 \to \R^2$ defined by the matrix under the identification $S^1 \times S^1 \cong \R^2 / \Z^2$. In particular, 
the map $h_K$ maps the four fixed points of $\sigma|_{\partial Y(K)}$ to the points $(0,0), (0,\frac{1}{2}),$ $(\frac{1}{2},0)$ and $(1/2,1/2)$ of $S^1 \times S^1 \cong \R^2 / \Z^2$. 
\end{itemize}
\end{lemma}

\begin{proof}
The existence of (two) invariant meridians is clear -- each of these contain two of the intersection points of the axis of $\sigma$ with $\partial Y(K)$. For the existence of invariant longitudes we consider one of the two halves that we obtain from $N(K)$ by cutting along the two invariant meridional disks. On the cylindrical part of its boundary we connect an intersection point of the axis with one meridian to an intersection point of the axis with the other meridian in such a way that the `partial linking number' is 0. We can do this because from a given starting point we have two possibilities for the endpoints on the other meridian. Then we apply the involution $\sigma$ to this curve, and the union of the two now is a longitude of $K$. We omit the rest of the proof which is straightforward.
\end{proof}

Torus knots have Seifert fibered complements. In fact, the 3-sphere has a well-known Seifert fibration with two multiple fibers, of order $p$ and $q$. A regular fibre of this Seifert fibration is a $T_{p,q}$ torus knot. We shall also allow torus knots which are the unknot, i.e. the torus knot $T_{3,1}$, if we want to make allusion to the corresponding Seifert fibration of its complement.

\begin{lemma}
   For torus knots $T_{p,q}$ arising as the regular fibre of a Seifert fibration of $S^3$, we can find an involution $\sigma$ which strongly inverts the torus knot and preserves the Seifert fibered structure of the complement. In particular, we can find a tubular neighborhood $N(T_{p,q})$ which is $\sigma$-invariant and which respects the Seifert fibered structure.
\end{lemma}

\begin{lemma}
  Suppose $N(T_{p,q})$ is a regular neighborhood of $T_{p,q}$ which is chosen $\sigma$-invariant and which preserves the Seifert fibered structure. Under an identification $h: \partial Y(T_{p,q}) \to S^1 \times S^1$ as in Lemma \ref{parametrisation} above, the Seifert fibers correspond to lines of slope $pq$ under the identification $S^1 \times S^1 \cong \R^2/\Z^2$. In particular, a Seifert fibre winds $pq$ times around the meridian while once around the longitude. 
\end{lemma}

%If we compose such a map $h$ further with
% \[
% S_{p,q}:= \begin{pmatrix} 1 & 0 \\ pq & 1 \end{pmatrix} : S^1 \times S^1 \, , 
% \]

We consider the map $s_{p,q}: S^1 \times S^1 \to S^1 \times S^1$ defined by the matrix
 \[
 s_{p,q}:= \begin{pmatrix} 1 & 0 \\ -pq & 1 \end{pmatrix} \, .
 \]
It has the property that 
\[ s_{p,q} \circ h_{T_{p,q}}: \partial Y(T_{p,q}) \to S^1 \times S^1
\]
 sends Seifert fibers to the first $S^1$ factor, and meridians to the second $S^1$ factor, and is orientation preserving. 

We will consider 3-manifolds that we obtain from glueing two torus knot complements together in such a way that the Seifert fibre of the first component is mapped to the meridian of the second component, and that the meridian of the first component is mapped to the Seifert fibre of the second component. In other words, given identifications of the boundary of each torus knot complement with $S^1 \times S^1$, such that the first factor corresponds to Seifert fibers, and such that the second component corresponds to meridians, our glueing is described by the matrix
\[
\begin{pmatrix} 0 & 1 \\ 1 & 0 \end{pmatrix} \, .
\]
Notice that this is orientation-reversing. 

\begin{definition}
Let $\phi: \partial Y(T_{p,q}) \to \partial Y(T_{r,s})$ be the orientation-reversing homeomorphism defined by
\[
	\phi:= h_{T_{r,s}}^{-1} \circ s_{r,s}^{-1} \circ \begin{pmatrix} 0 & 1 \\ 1 & 0 \end{pmatrix} \circ s_{p,q} \circ h_{T_{p,q}} \, .
\]
In other words, $\phi$ maps Seifert fibers of $T_{p,q}$ to meridians of $T_{r,s}$ and meridians of $T_{p,q}$ to Seifert fibers of $T_{r,s}$. We define the closed 3-manifold $Y(T(p,q),T(r,s))$ by the glueing of $ Y(T_{p,q}) $ and $Y(T_{r,s})$ along their boundary via the homeomorphism $\phi$, 
\[
	Y\left(T(p,q),T(r,s)\right):= Y(T_{p,q}) \cup_\phi Y(T_{r,s}) \, .
\]
\end{definition}

\begin{prop}\label{presentation}
   For any four numbers $p,q,r,s \in \Z$, the first homology group $H_1(Y(T(p,q),T(r,s));\Z)$ has a presentation matrix given by 
\[
\begin{pmatrix} 0 & 1 \\ -rspq +1  & rs \end{pmatrix} \, .
\]
In particular, if $1-rspq \neq 0$, the 3-manifold $Y(T(p,q),T(r,s))$ is a rational homology sphere, and its order is given by $\abs{1-rspq}$. 
\end{prop}
\begin{proof}
This follows easily from the Mayer-Vietoris long exact sequence.
\end{proof}
\begin{prop}\label{double cover}
The manifold $Y(T(p,q),T(r,s))$ is a branched double cover of a knot or 2-component link in $S^3$ that we denote $L(T(p,q),T(r,s))$. If $1-rspq$ is odd this is a knot; otherwise it is a 2-component link. \end{prop}
\begin{proof}
If we denote by $\sigma_{p,q}$ and $\sigma_{r,s}$ the strong involutions of the torus knots $T(p,q)$ and $T(r,s)$, then the glueing map $\phi$ above interchanges these actions. This follows from Lemma \ref{parametrisation} above:  With the given parametrisation by $S^1 \times S^1$ of the boundary, the involutions are represented by the matrix 
\[
\begin{pmatrix} -1 & 0 \\ 0 & -1 \end{pmatrix}
\]
which commutes with the glueing map 
\[
s_{r,s}^{-1} \circ \begin{pmatrix} 0 & 1 \\ 1 & 0 \end{pmatrix} \circ s_{p,q} \ .
\]
Therefore, the two involutions extend to an involution $\sigma$ of $Y=Y(T(p,q),T(r,s))$. 
By Lemma \ref{knot complement quotient} above the quotient of $Y$ by this involution is indeed the 3-sphere. The fixed point loci on either side -- $Y(T(p,q))$ or $Y(T(r,s))$ -- are given by two arcs. These glue together either to a 2-component link or to a knot. 

A first argument concerning the number of components is explicit:
On the boundary of $Y(T(p,q))$, two Seifert fibers interchanged by the involution split the boundary torus into two halves, two annuli in fact. If one of $p$ and $q$ is even one can easily see that either fixed point arc in $Y(T(p,q))$ starts and ends at the same component of such a splitting of the boundary torus. (And in fact, one arc connects two points of one component, and the other connects two points of the other component). On the other hand, if both $p$ and $q$ are odd then either arc starts and ends at different components. The glueing prescription now connects the two arcs in such a way that the only possibility to obtain a 2-component link is when $p,q,r$ and $s$ are all odd.

A second argument is homological: For any link $L$ with $l$ components one has
\begin{equation}\label{number of components homological}
  \dim_{\Z/2} H_1(\Sigma_2(L);\Z/2) = l-1 \, .
\end{equation}
This formula follows like this: A presentation matrix of $H_1(\Sigma_2(L);\Z/2)$ is given by $V + V^t$, where $V$ is a Seifert matrix of $L$. See for instance \cite[Theorem 9.1]{Lickorish}. As we are working over $\Z/2$, we have $V + V^t = V - V^t$ which is a sum of blocks of the form
\[
 \begin{pmatrix} 0 & 1 \\ -1 & 0 \end{pmatrix}
\] 
followed by as many $0$'s on the diagonal as the number of components minus 1. 
\end{proof}

\begin{remark}
We would like to point out that the knot or link $L(T(p,q),T(r,s))$ is not determined by the homeomorphism type of $Y(T(p,q),T(r,s))$. It depends on the glueing map, which, in our description above, depends on some choices in the parametrisation of the boundary. Different choices may result in isotopic glueing maps with different identifications of the fixed points (though not any combination is possible). The link 
$L(T(p,q),T(r,s))$ is therefore only determined {\em up to mutation} by the two strong involutions of $T(p,q)$ and $T(r,s)$.
\end{remark}

\begin{remark}
	If $T(r,s)$ is the trivial knot then the manifold $Y(T(p,q),T(r,s))$ is the effect of  Dehn surgery on $T(p,q)$ with slope $\frac{1-rs pq}{-rs} = \frac{-1}{rs} + pq$. By results of Moser \cite{Moser}, these are in fact lens space surgery slopes.
\end{remark}
\smallskip
\begin{lemma}\label{graph}
	If both torus knots $T_{p,q}$ and $T_{r,s}$ are different from the unknot then the 3-manifold $Y(T(p,q),T(r,s))$ admits a non-abelian fundamental group. In particular, it is not a Lens space. \end{lemma}
\begin{proof}
The torus along which the two knot complements $Y(T_{p,q})$ and $Y(T_{r,s})$ were glued together is incompressible and separates the 3-manifold $Y(T(p,q),T(r,s))$ into the two torus knot complements. 
In particular, the inclusion of this torus is $\pi_1$-injective into both sides. By the normal form theorem for amalgameted product, see for instance \cite[Theorem 25]{Cohen}, the inclusion of either side into the glued up manifold $Y(T(p,q),T(r,s))$ is also $\pi_1$-injective. Hence the fundamental group of $Y(T(p,q),T(r,s))$ contains a non-abelian subgroup.
\end{proof}

\begin{remark}
	In fact, one can show that under the conditions of the last Lemma the 3-manifold $Y(T(p,q),T(r,s))$ is not a Seifert-fibered 3-manifold. It is in fact a graph 3-manifold with two Seifert-fibSeifert-fiberedred components in its JSJ-decomposition. 
\end{remark}
The results of this section now yield the following 
\smallskip
\begin{theorem} \label{only binary dihedral}
	The 3-manifold $Y(T(p,q),T(r,s))$ comes with an involution with quotient $S^3$. It is a branched double cover of some knot or 2-component link $L(T(p,q),T(r,s))$ in $S^3$, well defined up to mutation by the involution on either side of the essential torus in $Y(T(p,q),T(r,s))$. If $pqrs-1$ is odd then $L(T(p,q),T(r,s)$ is a knot. If in addition both $T(p,q)$ and $T(r,s)$ are non-trivial torus knots, then the knot $L(T(p,q),T(r,s))$ is $SU(2)$-simple, but is not a 2-bridge knot.
\end{theorem}

\begin{proof}
   To simplify notations, we shall write $Y_0$ for $Y(T(p,q))$ and $Y_1$ for $Y(T(r,s))$ and simply $Y$ for the closed manifold $Y(T(p,q),T(r,s))$. We fix a base point on the torus along which $Y_0$ and $Y_1$ were glued together. We choose a meridian  $m_0$ of $T_{p,q}$ on the boundary of $Y_0$ that passes through the base point and a meridian $m_1$ for $T_{r,s}$ which also passes through the base point. We shall also denote by $m_0$ and $m_1$ the corresponding elements in the fundamental group $G_0:= \pi_1(Y_0)$ and $G_1:= \pi_1(Y_1)$. Likewise, we denote by $s_0$ a Seifert fibre at the boundary of $Y_0$, passing through the base point, and also by $s_0$ the corresponding element of $G_0$, and we define $s_1$ analogously. The group $G_i$ is normally generated by $m_i$ for $i=0,1$. The element $s_i$ lies in the centre of $G_i$. The fundamental group $G:=\pi_1(Y)$ is an amalgamated product of $G_0$ and $G_1$ over $\Z^2$. \\
   
That $Y$ is $SU(2)$-cyclic follows from Motegi's main result in \cite{Motegi}. We give a proof of this fact for the sake of completeness: Let $\rho: G \to SU(2)$ be a representation of the fundamental group of $Y$. We claim that $\rho$ has abelian image in $SU(2)$. Suppose this were not the case. Suppose first that the restriction of $\rho$ to the image of $G_0$ in $G$ were non-abelian [In fact $G_0$ injects into $G$ as by Dehn's Lemma the image of the boundary torus into each knot complement is $\pi_1$-injective. But we don't need this.] As $s_0$ is central in $G_0$ this implies that $\rho(s_0)$ lies in the centre $Z(SU(2)) = \{\pm 1\}$ of $SU(2)$. As $s_0=m_1$ we also have that $\rho(m_1)$ is central in $SU(2)$. As $s_1$ lies in the normal subgroup generated by $m_1$ in $G_1$, this implies that in fact $\rho(s_1)$ is central in $SU(2)$, hence also $\rho(m_0)$ is central in $SU(2)$ as $m_0 = s_1$. But as $G_0$ is normally generated by $m_0$ this contradicts the assumption that $\rho$ is non-abelian when restricted to $G_0$. 

Hence the restriction of $\rho$ to the image of $G_0$ or $G_1$ in $G$ is abelian. Next we show that this implies that $\rho$ has abelian image. 
If $\rho(m_0)$ is central we are done. 
Suppose $\rho(m_0)$ were non-central. Then $\rho(s_0)$ must commute with $\rho(m_0)$. If $\rho(s_0)$ is central then $\rho(m_1)$ is central and we are done. If $\rho(s_0)$ is not central it lies in the same maximal torus as $\rho(m_0)$ in $SU(2)$, and hence also $\rho(m_1)$ lies in the same maximal torus. Hence the image of $\rho$ is abelian in $SU(2)$. \\

If $1-pqrs$ is odd then $Y(T(p,q),T(r,s))$ is the branched double cover of the knot $L(T(p,q),T(r,s))$ by Proposition \ref{double cover}. The knot cannot be a 2-bridge knot because $Y(T(p,q),T(r,s))$ is not a Lens space by Lemma \ref{graph} above. That this knot admits only binary dihedral meridian-traceless $SU(2)$ representations follows now from Proposition \ref{folk} and Lemma \ref{cyclic vs abelian} above.
\end{proof}

\section{Examples}

We have computed the Khovanov homology and other knot invariants for `small' examples of the knots $L(T(p,q),T(r,s))$. In some cases, we were able to identify them in the knot table up to 12 crossings by their Alexander and Jones polynomial, using Cha-Livingston's ` Knotinfo' \cite{Knotinfo}.

\begin{equation*}
\begin{tabular}{l|c}
Knot  & Knot in table \\
\hline
$L(T(2,3),T(2,3))$ & $8_{16}$ \\
$L(T(2,3),T(2,-3))$ & $8_{17}$ \\
$L(T(2,3),T(2,5))$ & $9_{32}$ \\
$L(T(2,3),T(2,-5))$ & $9_{33}$ \\
$L(T(2,3),T(2,7))$ & $10_{83}$ \\
$L(T(2,3),T(2,-7))$ & $10_{86}$ \\
$L(T(2,5),T(2,5))$ & $10_{89}$ \\
$L(T(2,5),T(2,-5))$ & $10_{88}$ \\
$L(T(2,5),T(2,7))$ & $11_a54$\\
$L(T(2,7),T(2,7))$ & $12_a1010$ .
\end{tabular}
\end{equation*}

\smallskip

\section{Diagrams, alternatingness}\label{diagrams}
For rational tangles we follow the notations and conventions of \cite{Montesinos_Orsay}. In particular, a sequence of integers $(a_0, \dots, a_n)$ defines a rational tangle, with the convention used in this reference. By $\overline{(a_0, \dots, a_n)}$ we denote the tangle $(a_0, \dots, a_n)$ to which we apply a rotation by angle $\pi$ through the `vertical axis'. For example, the two diagrams below denote the tangle $(3,-4)$ and $\overline{(3,-4)}$. 

\begin{figure}[h!]
\def\svgwidth{0.4\columnwidth}
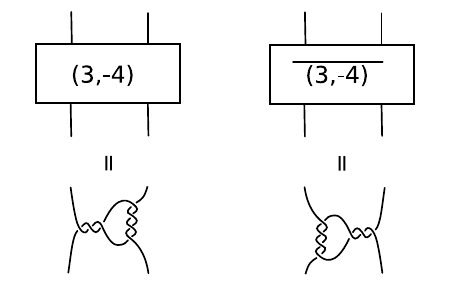
\end{figure}

\begin{figure}[h!]
\caption{The tangle $\tau(T(p,q))$ with $p/q = [a_0,a_1, \dots, a_n]$}
\label{tanglesum}
\def\svgwidth{\columnwidth}
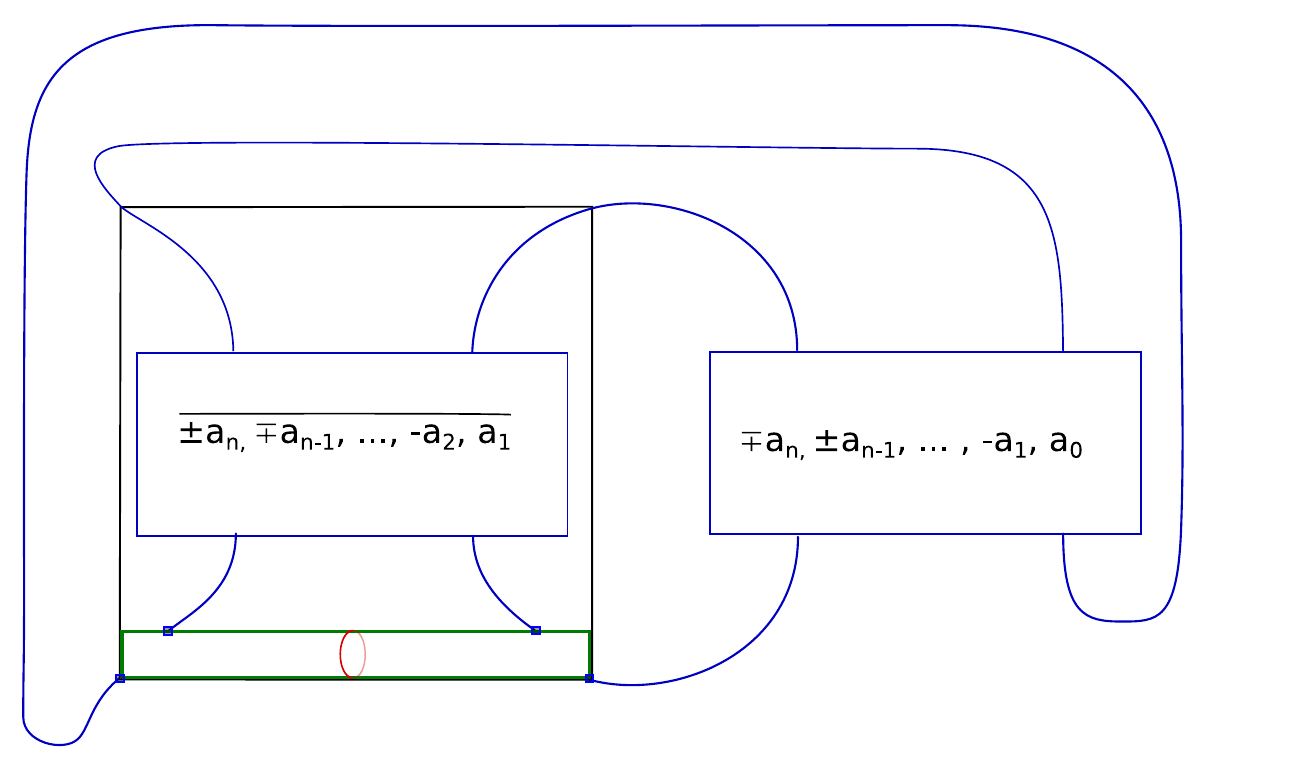
\end{figure}

\begin{theorem}\label{tangle_standard_diagram}
	If the fraction $p/q$ has continued fraction expansion
	\[
		\frac{p}{q} = a_0
          + \cfrac{1}{a_1
          + \cfrac{1}{a_2 + \cfrac{1}{\dots + \cfrac{1}{a_n} } } }   =: [a_0, a_1, \dots, a_n]\, ,
	\]
	one has $\abs{p/q} > 1$, and $n$ is odd, 
	then the tangle $\tau(T(p,q))$ has a diagram given by Figure
	\ref{tanglesum}
	 above. Notice that the two rational tangles appearing there always have opposite sign. In particular, the given diagrams are never alternating.
\end{theorem}
\begin{proof}
  The proof follows from the method in the proof of Theorem 1 in Montesinos' Orsay lecture notes \cite{Montesinos_Orsay} to which we refer the reader. The essential point is that we may write a linear map of the torus $T^2 = \R^2 / \Z^2$ of the form
  \[
  	\begin{pmatrix}  p & * \\ q & * \end{pmatrix} \in Sl(2,\Z)
  \]
as a product of matrices of the form 
\[
	\begin{pmatrix}  1 & a \\ 0 & 1 \end{pmatrix} \ \mbox{or} \ \begin{pmatrix}  1 & 0 \\ b & 1 \end{pmatrix}  \, .
\]
The latter maps are given as $a$ horizontal respectively $b$ vertical Dehn-twists of the torus. The precise relationship is via the claimed continued fraction expansion as shown in the reference op. cit. A full Dehn-twist of the torus induces a half-twist in the pillowcase quotient.
\end{proof}

\smallskip 
Concrete examples are given below. First, two general remarks are in order.

\begin{remark}
\begin{enumerate}
\item[(i)]
If $\abs{p/q} < 1$, and $n$ is even, then the role of the rational tangle inside the pillowcase and outside of it are interchanged.
\item[(ii)]
The continued fraction expansion is not unique, and in particular it is of no restriction to the generality to assume that $n$ is odd in the preceding statement. See the  remark at the end of the proof of Theorem 1 in \cite{Montesinos_Orsay}.
\end{enumerate}
\end{remark}
\smallskip 

\begin{example}
The tangle $\tau(T(16,5))$ is given by the Figure 
\ref{pillowcase_t165}
 below. Here the quotient of the tubular neighborhood of the torus knot $T(16,5)$ winds around the pillowcase. The quotient of a Seifert fibre is shown by the green curve. An isotopy brings this into the diagram shown in Figure 6
\ref{tangle_t165} 
 which has the shape as stated in Thoerem \ref{tangle_standard_diagram} above. In fact, we have 
\[
	\frac{16}{5} = 3 + \frac{1}{5} \, ,
\]
and the isotopy is given by first applying 3 `horizontal half-twists', fixing the lower side of the pillowcase and rotating the upper (thereby not changing the tangle inside the pillowcase), and then applying 5 `vertical half-twists', fixing the left side of the pillowcase, and rotating the right one.
\end{example}
\begin{figure}[h!]
\caption{The tangle $\tau(T(16,5))$ with the complementary ball wrapped around the pillowcase}
\label{pillowcase_t165}
\def\svgwidth{0.7\columnwidth}
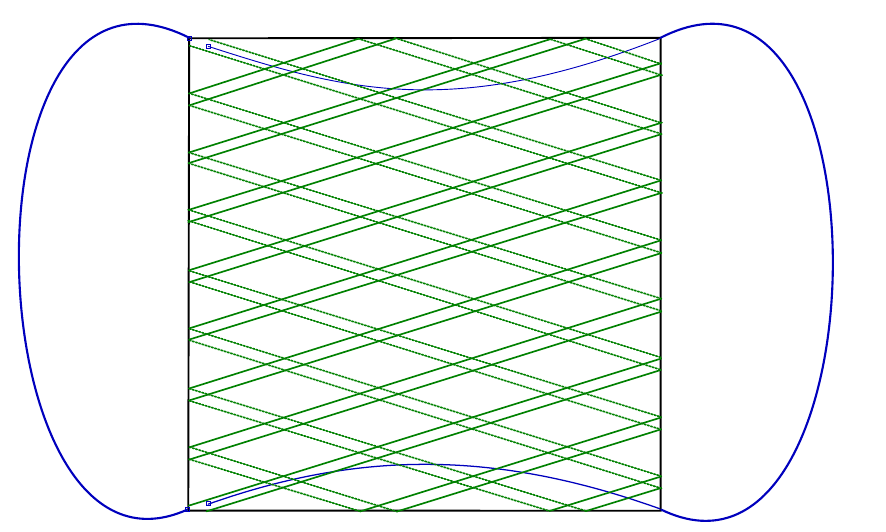
\end{figure}

\begin{figure}[hbt]
\caption{The tangle $\tau(T(16,5))$ depicted as stated in Theorem \ref{tangle_standard_diagram}. The complementary ball is isotoped into the lower end of the pillowcase.}
\label{tangle_t165}
\def\svgwidth{0.5\columnwidth}
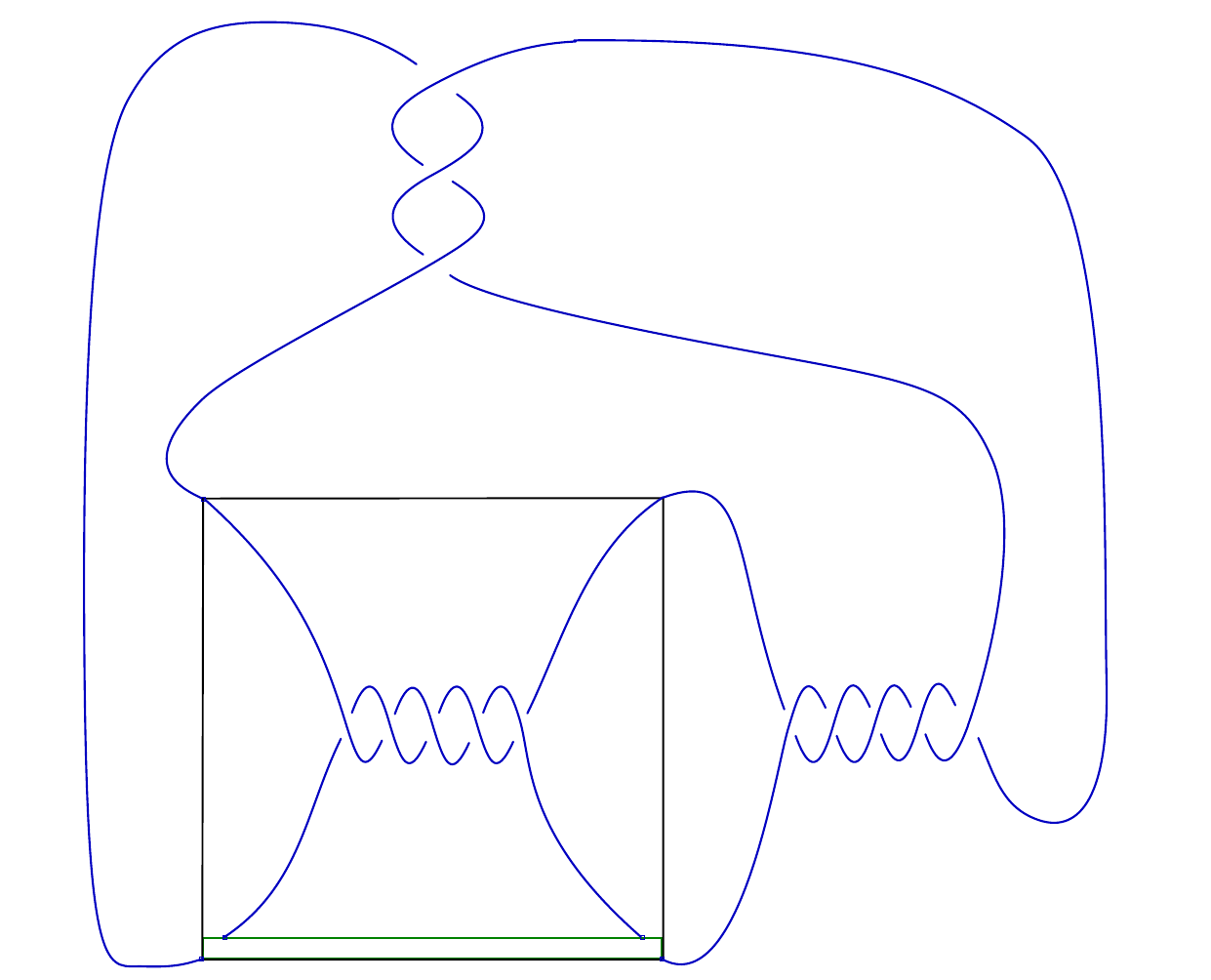
\end{figure}

\begin{example}
A general tangle $\tau(T(p,q))$ has a diagram given by Figure \ref{tanglesum_example}
if the length of the continued fraction expansion of $p/q$ is 4, $p/q = [a_0, \dots, a_3]$. 
\end{example}

\begin{figure}[h!t]
\caption{An example with length 4 in the continued fraction expansion.}
\label{tanglesum_example}
\def\svgwidth{0.7\columnwidth}
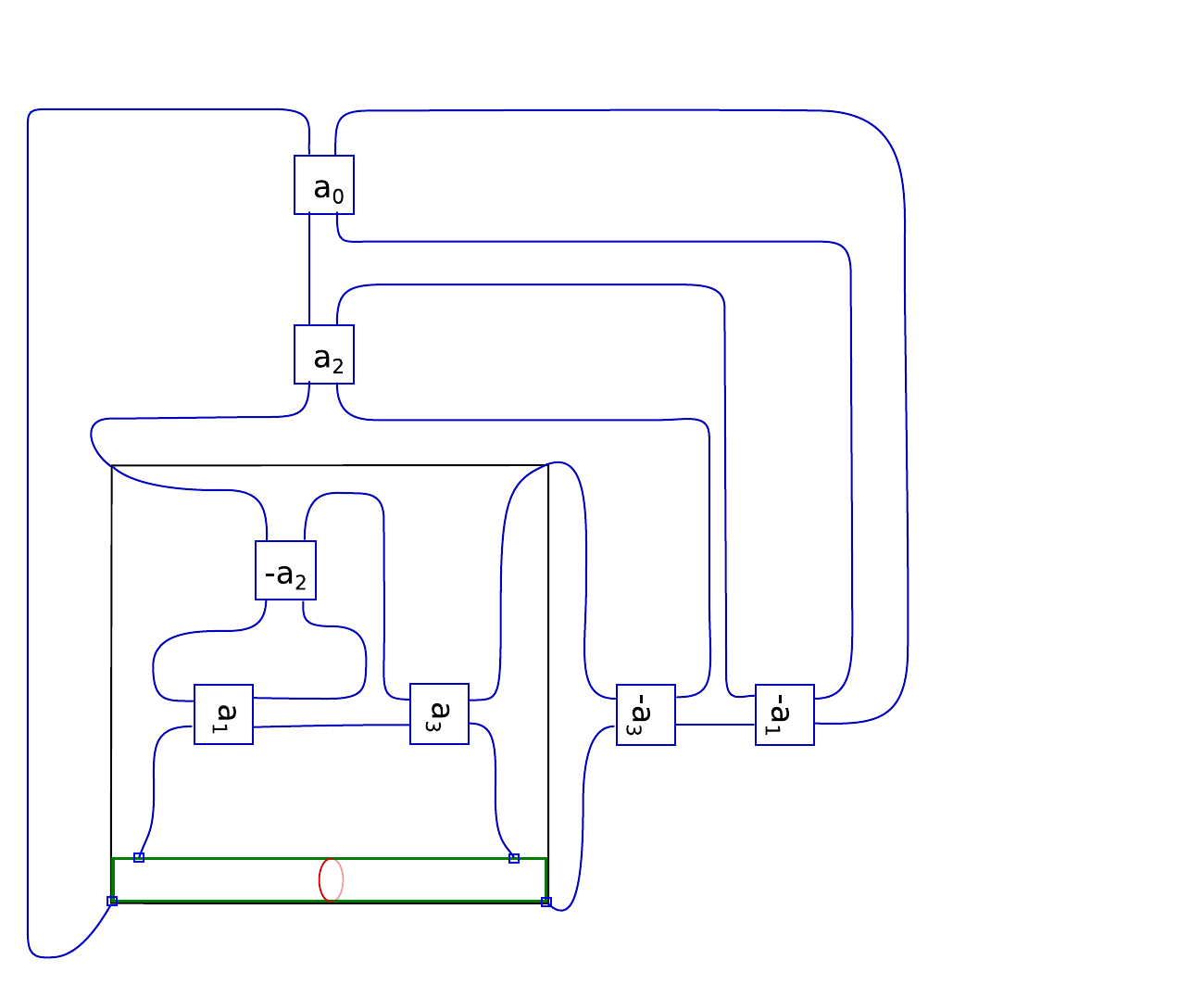
\end{figure}

\vspace{1cm}
%\newpage

In Figure \ref{alternating1} below the Conway sphere giving rise to the essential torus in the branched double cover is indicated by a grey dotted circle. A Conway sphere which can be indicated in such a way (by an embedded circle in the plane cutting the link diagram in four points) is called `visible' in the diagram in the terminology of Thistlethwaite, see \cite{Thistlethwaite}. The Conway sphere indicated by the dotted line in Figure \ref{alternating3} below is called `hidden' in Thistlethwaite's terminology.

\begin{theorem}\label{alternatingness}
	The knots $L(T(p,q),T(r,s))$ are alternating. The Conway sphere giving rise to the essential torus in the branched double cover is not visible in any alternating diagram of the link. 

\end{theorem}

\begin{proof}
We consider the case where $p/q$ and $r/s$ both have the same sign. The other case is similar and is left as an exercise.

\begin{figure}[hbt]
\caption{The initial diagram. The grey dotted line indicates the Conway sphere.}
\label{alternating1}
\def\svgwidth{0.7\columnwidth}
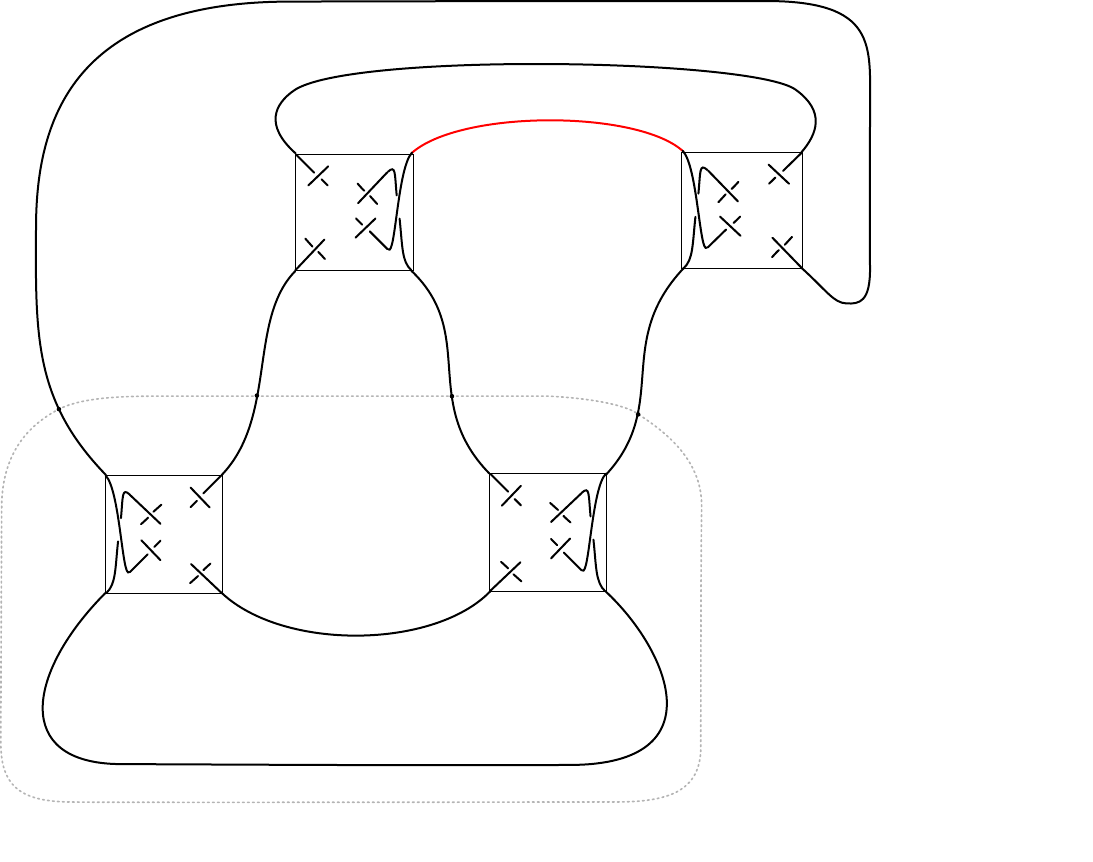
\end{figure}

By Theorem \ref{tangle_standard_diagram} above the knot $L(T(p,q),T(r,s))$ has a diagram given by Figure \ref{alternating1}. Here we have only indicated the information relevant to alternatingness. In particular, each square contains a rational tangle with an alternating diagram, because there is a continued fraction expansion $p/q = [a_0, \dots, a_n]$ with either $a_i > 0$ for $ i=0, \dots, n$ or $a_i < 0$ for $ i=0, \dots, n$, and likewise for $r/s$. It may be drawn such that all overcrossings go either `from upper left to lower right' or `from lower left to upper right'. In each of these rational tangles, some indicated crossings may be identical. For instance, in the upper right square, the two strands entering from the upper right and lower right side may have their first crossing together in the same way as the two strands entering from the upper left and lower left side.

\begin{figure}[h!]
\caption{An intermediate diagram.}
\label{alternating2}
\def\svgwidth{0.7\columnwidth}
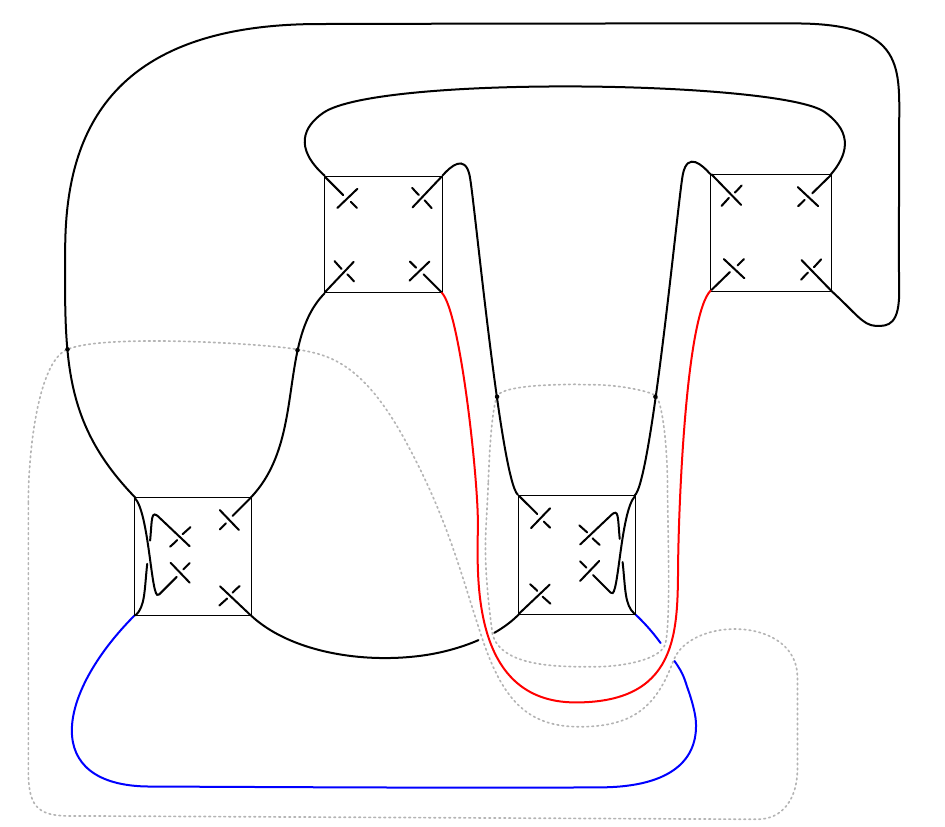
\end{figure}

A first isotopy takes the red strand and pulls it over the lower right square. This yields the diagram in Figure \ref{alternating2}. It still has the same number of crossings as the initial diagram. 

A second isotopy takes the blue strand and pulls it underneath the upper right square. This yields an alternating diagram with two crossings less as the initial diagram, depicted in Figure \ref{alternating3}.

\begin{figure}[h!]
\caption{The final alternating diagram. The Conway sphere is hidden.}
\label{alternating3}
\def\svgwidth{0.7\columnwidth}
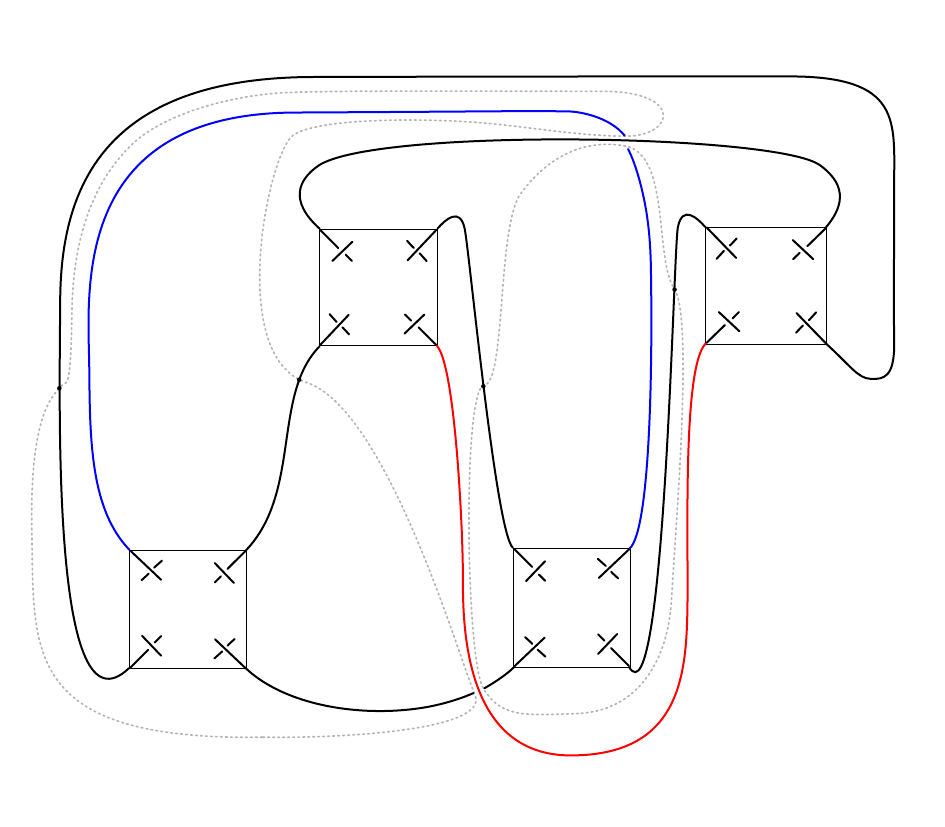
\end{figure}

In Figure \ref{alternating1} the Conway sphere giving rise to the essential torus in the branched double cover indicated by the grey dotted line is visible.
It is isotopic to the Conway spheres indicated in Figure \ref{alternating2} and Figure \ref{alternating3}. The Conway sphere in the last figure is  `hidden'. 

Menasco \cite{Menasco} has shown that a Conway sphere in an alternating diagram is isotopic to one that is either visible or hidden. Building on this, Thistlethwaite \cite{Thistlethwaite} has shown that this is a characteristic property of the link in the following sense: If a Conway sphere is hidden in an alternating diagram of a link, then it is hidden in any alternating diagram of the link. 
\end{proof}

\begin{remark}
	It seems that the Conway sphere lifting to the incompressible torus in the branched double cover $Y(T(p,q),T(r,s)) = \Sigma_2(L(T(p,q),T(r,s)))$ cannot be seen in the `obvious way' in the alternating diagram of $L(T(p,q),T(r,s))$ that we have described. By the `obvious way' we mean that it cannot be described by a circle in the diagram that cuts the diagram in precisely four points, as is the case, for instance, for the Conway-knot and the Kinoshito-Terasaka knot. 
\end{remark}

\begin{remark}
	It is clear that the above proof applies to more general knots containing four rational tangles. This could be of possible interest to the methods used in \cite{Greene_Levine}. 
\end{remark}

\section{Relevance to instanton knot Floer homology}
The instanton knot Floer homology $I^\natural(K)$ of a knot $K$, as defined by Kronheimer and Mrowka in \cite{KM_ss}, has close ties with the representation varieties 
$R(K,SU(2);\i)$ considered in this paper. For the precise setup we refer to this reference. In the following, we shall only review that part of the construction which is relevant to our work on representation varieties.

Associated to $K$ there is a 2-component link $K^\natural$, consisting of $K$ together with the boundary $\mu$ of a meridional disc, and an arc $\omega$ connecting these two components, see \cite[Figure 3]{KM_ss}. There is an associated space of representations
\[
	R^\natural(K) = \{ \rho \in \Hom(\pi_1(S^3 \setminus (K^\natural \cup \omega)), SU(2)) | \rho(m_\mu) \sim \i, \rho(m_K) \sim \i, \rho(m_\omega) = -1 \} , 
\]
where $m_K, m_\mu$,and $m_\omega$ denote small meridians in $S^3 \setminus (K^\natural \cup \omega)$ of $K, \mu$, and the arc $\omega$. There is a map induced by restriction 
\begin{equation} \label{two representation spaces}
	R^\natural(K)/SU(2) \to R(K,\i)/SU(2)  , 
\end{equation}
where the action of $SU(2)$ on both sides is by conjugation.
\begin{prop}\cite[Proposition 4.3]{HHK}\label{map_rep_spaces}
The map (\ref{two representation spaces}) is surjective. The preimage of a conjugacy class of a non-abelian representation is a circle of non-abelian representations, and the preimage of the conjugacy class of the abelian representation is a point. Furthermore, $R^\natural(K)$ consists only of non-abelian representations.
\end{prop}
The reduced instanton Floer homology $I^\natural(K)$ is a $\Z/4$-graded abelian group associated to a knot or link $K$. 
It is the Morse homology in an appropriate sense of a function $\text{\em CS}$ defined on a certain space of connections with prescribed singularities around $K^\natural \cup \omega$, and whose critical manifold consists of certain flat connections and is identified with $R^\natural(K)/SU(2)$, the identification being determined by the holonomy representation as usual in such a setup. For non-trivial knots, the critical manifold always has positive dimensional components. 

Therefore, the function $\text{\em CS}$ has to be perturbed to a function $\text{\em CS}_\pi$ with the desired properties: The resulting critical manifold $R^\natural_\pi(K)/SU(2)$ is a finite set of points, and the moduli space of flow-lines in this setup satisfies a certain transversality condition analogous to the Morse-Smale condition in the classical setup of Morse homology. That both can be achieved is proved in \cite{KM_ss,KM_knots}, using suitable holonomy perturbations. Such holonomy perturbations are called {\em admissible}. Instanton Floer homology is then the homology of a $\Z/4$-graded complex $CI^\natural_\pi(K)$ with generators given by the points of $R^\natural_\pi(K)/SU(2)$, for an admissible perturbation $\pi$. 
\\

Provided $R^\natural(K)/SU(2)$ is already non-degenerate in the Morse-Bott sense, one expects that it is possible to perturb in such a way that every circle of $R^\natural(K)/SU(2)$ is replaced by exactly two critical points in $R^\natural_\pi(K)/SU(2)$, and that the `abelian point' in $R^\natural(K)/SU(2)$ corresponds to a unique point in $R^\natural_\pi(K)/SU(2)$. 

Kronheimer and Mrowka have given a cohomological criterion to when a representation of $R^\natural(K)/SU(2)$ is non-degenerate in the Morse-Bott sense. In fact, a representation $\rho \in R^\natural(K)$ is non-degenerate if and only if the restriction map
\begin{equation}\label{morse-bott}
	H^1(Y \setminus K^\natural; \su(2)_\rho) \to H^1(m_K; \su(2)_\rho) \oplus H^1(m_{\mu};\su(2)_\rho)
\end{equation}
has trivial kernel, see \cite[Lemma 3.13]{KM_knots}. Here the cohomology groups are considered to be with twisted coefficients, and where $\su(2)$ becomes a $\pi_1$-module via the representation $\rho$. The description in Proposition \ref{map_rep_spaces} is very explicit. In fact, an irreducible representation $\rho \in R(K,\i)$ gives rise to a circle of representations in $R^\natural(K)$ which map the meridian $\rho(m_\mu)$ to any element orthogonal to $\rho(m_K) \sim \i$ and of norm $1$ in the purely imaginary quaternionians (there is a circle worth of these, and if $\rho(m_K) = \i$, then this circle is the unit circle in the plane spanned by $\j, \k$). 

With this description, and with a Mayer-Vietoris argument one can see that an irreducible representation $\rho \in R(K,\i)$ gives rise to a circle in $R^\natural(K)/SU(2)$ satisfying the Morse-Bott non-degeneracy assumption if and only if the twisted cohomology group $H^1(Y \setminus K; \su(2)_\rho)$ is one-dimensional, and the restriction map 
\begin{equation}\label{morse-bott-classic}
	H^1(Y \setminus K; \su(2)_\rho) \to H^1(m_K; \su(2)_\rho)
\end{equation}
is onto. Intuitively this condition means the following: Thinking of the restriction map
\[
\Hom(\pi_1(Y(K)),SU(2))/SU(2) \to \Hom (\pi_1(\partial Y(K)),SU(2))/SU(2) =:P \, ,
\]
where $P$ is the 2-dimensional pillowcase, one can deform the representation $\rho$ (inside the variety of all representations, not requiring to send the meridian to an element of trace 0) in such a way that this yields to a deformation inside $P$ which is transverse to the section of elements required to send the meridian to elements of trace 0. (This is where the image of $R(K,\i)$ maps to inside $P$).

\begin{hypothesis}\label{hypothesis regularity}
Suppose the knot $K$ is such that $R(K,\i)/SU(2)$ contains exactly $n$ conjugacy classes of irreducible representations (and necessarily a single conjugacy class of reducible representations) each having one-dimensional twisted cohomology group $H^1(Y \setminus K; \su(2)_\rho)$, and such that the map (\ref{morse-bott-classic}) is onto. 
Then there is a sufficiently small admissible holonomy perturbation $\pi$ of the Chern-Simons function such that $R^\natural_\pi(K)/SU(2)$ contains exactly $2n+1$ points. As a consequence, the associated instanton Floer chain complex $CI^\natural_\pi(K)$ is of total rank $2n+1$. 
\end{hypothesis}

The point of this hypothesis is that the perturbation gives rise to precisely two critical points out of each circle satisfying the Morse-Bott condition. The non-degeneracy assumption is an open condition, so for sufficiently small perturbations the critical points will still be non-degenerate, from which the claim about the Floer chain complex follows. 

Hedden, Herald and Kirk have proved in \cite{HHK} that this hypothesis is satisfied for all 2-bridge knots, torus knots, and certain other classes of knots. Further examples have been studied by Fukumoto, Kirk and Pinz\'on-Caicedo in \cite{FKP}. We expect that our knots $L(T(p,q),T(r,s)$ also satisfy this hypothesis. We plan to come back to this in forthcoming work.
\\

The relationship of instanton knot Floer homology to the Alexander polynomial, established independently by  Kronheimer and Mrowka in \cite{KM_Alex} and Lim in \cite{Lim}, gives the lower bound $\det(K) = \abs{\Delta_K(-1)}$ to the rank of instanton Floer homology, so that Hypothesis \ref{hypothesis regularity} yields the following

\begin{prop}
If a knot $K$ is $SU(2)$-simple and satisfies the genericity hypothesis \ref{hypothesis regularity} above, then its instanton Floer chain complex $CI^\natural_\pi(K)$ has no non-zero differentials and is of total rank the $\det(K)$. In particular, the total rank of reduced instanton knot Floer homology $I^\natural(K)$ is also equal to $\det(K)$. 
\end{prop}
\begin{proof}
For a knot $K$, the number of conjugacy classes of binary dihedral representations in $R(K;\i)$ is equal to 
	$(\det(K)-1)/2$, by Klassen's theorem \cite{Klassen}. Therefore, an $SU(2)$-simple knot $K$ satisfying the genericity hypothesis has instanton Floer chain complex of rank $\det(K)$. For a knot $K$, instanton knot Floer homology $I^\natural(K)$ is isomorphic to a different flavor of instanton knot Floer homology $\text{\em KHI}(K)$ as defined in \cite{KM_sutures}. The latter categorifies the Alexander polynomial $\Delta_K(t)$ by the results in \cite{KM_Alex,Lim}. In particular, the rank of instanton Floer homology has to be greater or equal than the determinant $\det(K) = \abs{\Delta_K(-1)}$. Therefore, the complex $CI^\natural_\pi(K)$ cannot have a non-zero differential.
\end{proof}

For the class of $SU(2)$-simple knots of the form $L(T(p,q),T(r,s))$ defined in this paper we obtain the total rank of reduced instanton knot Floer homology without having to be concerned about whether the genericity hypothesis \ref{hypothesis regularity} is satisfied. One makes use of the Kronheimer-Mrowka spectral sequence which must be degenerate because $L(T(p,q),T(r,s))$ is alternating, see \cite{KM_ss}. 

%\begin{prop}
%	The total rank of reduced instanton knot Floer homology we have: 
%	\[
%	I^\natural(L(T(p,q),T(r,s))) \cong \Z^{\abs{pqrs-1}} \, .
%	\]
%	Its $\Z/4$-grading is determined from the Jones polynomial and the knot signature of $L(T(p,q),T(r,s))$.
%\end{prop}
%\begin{proof}
%Theorem \ref{alternatingness} together with the Kronheimer-Mrowka spectral sequence \cite{KM_ss} readily determines the total rank of reduced instanton knot Floer homology, since the Kronheimer-Mrowka spectral sequence must degenerate at the page isomorphic to Khovanov homology in this case. The Khovanov homology determines the $\Z/4$ grading, see \cite{KM_ss}. The Khovanov homology in turn is determined by the Jones polynomial and the knot signature, see Lee's theorem \cite{Lee}. Its total rank is equal to the knot determinant, which is equal to $\abs{pqrs-1}$ in this case.
%\end{proof}

\section{$SU(2)$-simple knots and Heegaard-Floer homology strong L-spaces}
The simplest version of Heegaard-Floer homology is the `hat'-version \cite{OzSz1,OzSz2}, introduced by Ozsv\'ath and Szab\'o. It is an abelian group $\widehat{\text{\em HF}}(Y)$, associated to a rational homology sphere $Y$. 
Such a manifold $Y$ is called a Heegaard-Floer homology L-space if the  $\widehat{\text{\em HF}}(Y)$ is of rank equal to the cardinality of $H_1(Y;\Z)$. As the latter number is always a lower bound to the rank of $\widehat{\text{\em HF}}(Y)$, one can say that an L-space has Heegaard-Floer-hat homology as small as it can possibly be. A rational homology sphere is called a Heegaard-Floer homology {\em strong} L-space if the rank of the complex $\widehat{\text{\em CF}}(Y)$ computing $\widehat{\text{\em HF}}(Y)$ is equal to the cardinality of $H_1(Y;\Z)$. Such a complex has no non-trivial differentials. An example of a 3-manifold which is an L-space, but not a strong L-space is the Poincar\'e homology sphere. 

A strong L-space $Y$ is particularly simple with respect to Heegaard-Floer-hat homology. It has a formal similarity to a $SU(2)$-simple knot $K$ satisfying the genericity assumption \ref{hypothesis regularity}. However, this formal resemblance is reflected topologically for the class of $SU(2)$-simple knots studied here:
\\

\begin{prop}
	The 3-manifolds $Y(T(p,q),T(r,s)) = \Sigma_2(L(T(p,q),T(r,s)))$ are Heegaard-Floer homology strong L-spaces.
\end{prop}
\begin{proof}
By Theorem \ref{alternatingness} the knots $L(T(p,q),T(r,s))$ are alternating. Their branched double covers $Y(T(p,q),T(r,s))$ are therefore Heegaard-Floer homology strong L-spaces by a result of Greene \cite{Greene}. 
\end{proof}

\begin{question}
  Is the branched double cover $\Sigma_2(K)$ of a $SU(2)$-simple knot $K$ always a strong L-space? Is every $SU(2)$-simple knot alternating? 
\end{question}
\begin{remark}  
The converse of the relationship asked in this question is not true, however. In fact, all pretzel knots $P(p,q,r)$ with $p,q,r > 1$ are alternating and therefore $\Sigma_2(P(p,q,r))$ are strong L-spaces \cite{Greene}, but these knots are not $SU(2)$-simple as was shown in \cite{Z}. In fact, these knots possess many representations in $R(P(p,q,r);\i)$ which are not binary dihedral. 
\end{remark}

%As a consistency check, one may realise that the weaker statement that the $Y(T(p,q),T(r,s))$ are Heegaard-Floer homology L-spaces follows easily from 
%Hanselman's result \cite[Theorem 2]{Hanselman}. In his notation,  our 3-manifold $Y(T(p,q),T(r,s))$ is denoted $Y(T(p,q)^{[pq]}, T(r,s)^{[rs]})$. Torus knots are L-space knots which means that they admit non-trivial surgeries which are L-spaces. In fact, torus knots admit a lot of Lens space surgeries, see \cite{Moser}. The last condition in Hanselman's theorem is fulfilled because the slope $pq$ is strictly bigger than twice the Heegaard-Floer concordance invariant of $T(p,q)$ if it is positive, and strictly smaller if it is negative. 
%\\

\section{Integer homology spheres and irreducible $SU(2)$ representations}
Kronheimer-Mrowka's non-vanishing result \cite[Corollary 7.17]{KM_sutures} of instanton knot Floer homology is the essential input in the following 
\begin{prop}
Let $Y$ be an integer homology 3-sphere which is the branched double cover of a non-trivial knot $K$ in $S^3$. Then there is an irreducible representation $\psi: \pi_1(Y) \to SU(2)$. 
\end{prop}
\begin{proof}
The crucial input is that there is an irreducible representation $\rho: \pi_1(S^3 \setminus K) \to SU(2)$ which maps a meridian $m$ to the element $\i \in SU(2)$. This was proved by Kronheimer-Mrowka in \cite[Corollary 7.17]{KM_sutures}. As in the proof of Proposition \ref{folk} above, we see that this representation induces a  representation with non-abelian image
\[
	\overline{\rho}: G_{K,m^2} \to SO(3) \, .
\]
The fundamental group $\pi_1(Y) = \pi_1(\Sigma_2(K))$ is contained in $G_{K,m^2}$ as a subgroup of index $2$. If the restriction of $\overline{\rho}$ to this subgroup had abelian image it would have trivial image, as $Y$ is a homology 3-sphere. Therefore, $\overline{\rho}$ would factor through $G_{K,m^2}/\pi_1(Y) \cong \Z/2$, contradicting the fact that $\overline{\rho}$ has non-abelian image. 

Hence the restriction of $\overline{\rho}$ to $\pi_1(Y)$ has non-abelian image, and hence lifts to an irreducible representation $\psi: \pi_1(Y) \to SU(2)$. 
\end{proof}

By results of Bonahon-Siebenmann every graph 3-manifold is a branched double cover of an arborescent link $L$, see \cite[Appendix A]{Bonahon-Siebenmann} and \cite[Section 1.1.8 and 1.1.9]{Saveliev}. If the branched double cover of $L$ is an integer homology sphere, then the link $L$ can have only one component by formula (\ref{number of components homological}). 
Therefore, we obtain the following

\begin{corollary}
Any graph 3-manifold $Y$ which is an integer homology sphere admits an irreducible representation $\rho: \pi_1(Y) \to SU(2)$.   
\end{corollary}
\qed

\section{Discussion, Perspectives}

\begin{question}
Is the 3-sphere the only integer homology sphere admitting only the trivial $SU(2)$ representation? It follows from results of Kronheimer-Mrowka that no integer homology sphere obtained by Dehn surgery on a non-trivial knot has only the trivial $SU(2)$ representation \cite{KM_Dehn,KM_p}. However, it is known that there are integer homology 3-spheres which cannot be obtained by surgery on a knot.
\end{question}
A Floer sphere $Y$ is a 3-manifold that has the same instanton Floer homology $I_*(Y)$ as the 3-sphere. There seems to be no example known of a Floer sphere other than the 3-sphere. Any integer homology sphere admitting only the trivial $SU(2)$ representation of its fundamental group is a Floer sphere. An affirmative answer to the above question  would therefore raise the possibility that instanton Floer homology (of homology 3-spheres) detects the 3-sphere.

%The examples of 3-manifolds $Y(T(p,q),T(r,s))$ we were considering above are graph manifolds if both torus knots $T(p,q)$ and $T(r,s)$ are non-trivial. It is therefore natural to ask the following question: 

\begin{question}
Are there hyperbolic 3-manifolds or `mixed' 3-manifolds which have only cyclic or only abelian $SO(3)$ or $SU(2)$ representations? 
\end{question}
The answer to this question is yes due to results of Cornwell \cite{Cornwell} which have appeared after the first posting of this article on the arXiv. After Corollary 6.7 of his paper, Cornwell lists the knot $8_{18}$ among other $SU(2)$-simple knots which are not 2-bridge. Hence by Theorem 1.1 of his paper ($8_{18}$ is a 3-bridge knot) all representations of the fundamental group of the branched double cover of $8_{18}$ are $SU(2)$-cyclic.  
But the branched double cover of the knot $8_{18}$ is hyperbolic. This knot is the Turk's head knot with notation $(2 \times 4)^*$ in \cite[Section 4.3]{Bonahon-Siebenmann}, and in this reference it is shown that its branched double cover is hyperbolic.

\begin{question}
Can the rational homology spheres $Y(T(p,q),T(r,s))$ be obtained by surgery on a knot? If this were the case, then this knot would admit surgery slopes which are in general both $SU(2)$-cyclic  and $SO(3)$-cyclic. Examples of such surgeries comprise the Dehn surgery with slope $37/2$ on the Pretzel knot $P(-2,3,7)$, according to Dunfield \cite{Dunfield}. This is not a Lens space surgery as all Lens space surgeries have integer surgery slope. It is claimed in \cite{Lin} that this example of Dunfield's is in fact the manifold which is $Y(T(3,2),T(-3,2))$ in our notation. 
\end{question}

An anonymous referee has drawn the author's attention to the following class of $SU(2)$-cyclic, but not $SO(3)$-cyclic surgeries.
\begin{example}\label{connected sum}
According to \cite[Proposition 4]{Moser}, surgery with coefficient $2n$ on the torus knot $T(2,n)$ is the manifold $L(2,1) \# L(n,2)$. This has only cyclic $SU(2)$-representation of its fundamental group, because the free factor $\Z/2$ can only map to the centre of $SU(2)$. However, these manifolds have irreducible $SO(3)$-representa\-tions. 
\end{example}

%\begin{question}
%Of all graph 3-manifolds, are the manifolds $Y(T(p,q),T(r,s))$ the only manifolds with only cyclic $SU(2)$ or $SO(3)$ representations?
%This question seems easier to answer than the question on hyperbolic 3-manifolds or `mixed' 3-manifolds.
%\end{question}


\begin{thebibliography}{99999}
\bibitem{Knotinfo} J. C. Cha and C. Livingston, {\em KnotInfo: Table of Knot Invariants, http://www.indiana.edu/~knotinfo,} November 2014
\bibitem{Boileau_Zieschang} M. Boileau, H. Zieschang, {\em Nombre de ponts et g\'en\'erateurs m\'eridiens des entrelacs de
   Montesinos}, Comment. Math. Helv.,
   60 (1985), no. 2, 270--279.
\bibitem{Bonahon-Siebenmann} F. Bonahon, L. Siebenmann, {\em New Geometric Splittings of Classical Knots
and
the Classification and Symmetries of
Arborescent Knots}, to appear in Geometry \& Topology Monographs
\bibitem{Cohen} D. Cohen, {\em Combinatorial group theory: a topological approach.}, London Mathematical Society Student Texts, 14. Cambridge University Press, Cambridge, 1989.
\bibitem{Collin-Saveliev} O. Collin, N. Saveliev, {\em Equivariant Casson invariants via gauge theory,} 
J. Reine Angew. Math. 541 (2001), 143--169. 
\bibitem{Cornwell} C. Cornwell, {\em Character varieties of knot complements and branched double-covers via the cord ring,} arXiv:1509.04962
\bibitem{Dunfield} N. Dunfield, private communication.
\bibitem{Floer}
A. Floer, {\em 
An instanton-invariant for 3-manifolds}, 
Comm. Math. Phys. 118 (1988), no. 2, 215--240. 
\bibitem{FKP} Y. Fukumoto, P. Kirk and J. Pinz\'on-Caicedo, {\em 
Traceless SU(2) representations of 2-stranded tangles}, arXiv:1305.6042 
\bibitem{Greene} J. Greene, {\em A spanning tree model for the Heegaard Floer homology of a branched double-cover}, J. Topol. 6 (2013), no. 2, 525--567. 
\bibitem{Greene_Levine} J. Greene, A. Levine, {\em Strong Heegaard diagrams and strong L-spaces,} preprint, arXiv:1411.6329
\bibitem{HHK} M. Hedden, C. Herald and P. Kirk, {\em The pillowcase and perturbations of traceless representations of knot groups, }  Geom. Topol. 18 (2014), no. 1, 211--287. 
\bibitem{Hedden-Kirk} M. Hedden, P. Kirk, {\em Instantons, concordance, and Whitehead doubling},  J. Differential Geom. 91 (2012), no. 2, 281--319.
\bibitem{Klassen} E. Klassen, {\em Representations of knot groups in SU(2)}, Trans. Amer. Math. Soc. 326 (1991), no. 2, 795--828. 
\bibitem{Khov} M. Khovanov, {\em A categorification of the Jones polynomial,} Duke Math. J. 101 (2000), no. 3, 359--426.
\bibitem{KM_p} P. Kronheimer, T. Mrowka, {\em Witten's conjecture and property P}, Geom. Topol. 8 (2004), 295--310.
\bibitem{KM_Dehn} P. Kronheimer, T. Mrowka, {\em Dehn surgery, the fundamental group and $SU(2)$}, Math. Res. Letters 11 (2004), no. 5-6, 741--754.
\bibitem{KM_knots} P. Kronheimer, T. Mrowka, {\em Knot homology groups from instantons}, J. Topol. 4 (2011), no. 4, 835--918. 
\bibitem{KM_sutures} P. Kronheimer, T. Mrowka, {\em Knots, sutures and excision}, J. Differential Geom. 84 (2010), no. 2, 301--364.
\bibitem{KM_Alex} P. Kronheimer, T. Mrowka, {\em Instanton Floer homology and the Alexander polynomial}, Algebr. Geom. Topol. 10 (2010), no. 3, 1715--1738. 
\bibitem{KM_ss} P. Kronheimer, T. Mrowka, {\em Khovanov homology is an unknot-detector},  Publ. Math. Inst. Hautes Etudes Sci. No. 113 (2011), 97--208.
\bibitem{Lee} E. Lee, {\em An endomorphism of the Khovanov invariant,}
Adv. Math. 197 (2005), No. 2, 554--586.
\bibitem{Lickorish} W. Lickorish, {\em Introduction to Knot theory}, Graduate Texts in Math. 175, Springer (1997)
\bibitem{Lim} Y. Lim, {\em Instanton homology and the Alexander polynomial}, Proc. Amer. Math. Soc.
138 (2010), 3759--3768. 
\bibitem{Lin} J. Lin, {\em $SU(2)$-cyclic surgeries on knots}, arXiv:1307.0070
\bibitem{LZ} A. Lobb, R. Zentner, {\em On spectral sequences from Khovanov homology}, preprint (2013)
\bibitem{Menasco} W. Menasco, {\em Closed incompressible surfaces in alternating knot and link complements,} Topology 23 (1984), no. 1, 37--44.
\bibitem{Montesinos_Orsay} J. Montesinos, {\em Rev\^etements ramifi\'es de noeuds, espaces fibr\'es de Seifert et scindements de Heegaard}, Lecture notes of a conference in Orsay, spring 1976. 
\bibitem{Moser} L. Moser, {\em Elementary surgery along a torus knot}, Pacific J. Math. 38 (1971), no. 2, 737--745.
\bibitem{Motegi} K. Motegi, {\em Haken manifolds and representations of their fundamental group in $SL(2,\C)$}, Topology and its Applications 29 (1988), 207--212.
\bibitem{OzSz1}  P. Ozsv\'ath, Z. Szab\'o, {\em Holomorphic disks and topological invariants for closed three-manifolds,} Ann. of Math. (2)
159 (2004), no. 3, 1027--1158.
\bibitem{OzSz2} P. Ozsv\'ath, Z. Szab\'o, {\em Holomorphic disks and three-manifold invariants: propert
ies and applications
,} Ann. of Math. (2) 159
(2004), no. 3, 1159--1245.
\bibitem{Saveliev} N. Saveliev, {\em Invariants for homology 3-spheres}, Encyclopaedia of Mathematical Sciences, Springer, Vol. 140 (2002)
\bibitem{Schreier} O. Schreier, {\em Über die gruppen $A^aB^b=1$,}
Abh. Math. Sem. Univ. Hamburg 3 (1924), no. 1, 167--169. 
%\bibitem{OzSz3} P.~Ozsv\'ath, Z.~Szab\'o, {\em On the Heegaard Floer homology of branched double-covers,} Adv. Math. 194 (2005), no. 1, 1--33.
\bibitem{Thistlethwaite} M. Thistlethwaite, {\em On the algebraic part of an alternating link,}
Pacific J. Math. 151 (1991), no. 2, 317--333. 
\bibitem{Z} R. Zentner, {\em Representation spaces of pretzel knots}, Algebraic \& Geometric Topology 11 (2011), no. 5, 2941--2970.

\end{thebibliography}
\end{document}